\documentclass[11pt,leqno]{amsart}
\usepackage{etex}
\usepackage[foot]{amsaddr}

\usepackage[utf8x]{inputenc} 	      
\usepackage{lmodern}                
\usepackage[T1]{fontenc}            
\usepackage[french,english]{babel}  
\usepackage{ucs}             	      

\usepackage{graphicx}         
\usepackage{xcolor}           
\usepackage{enumitem}         
\usepackage{tikz}
\usetikzlibrary{fit,calc,positioning,decorations.pathreplacing,matrix, shapes}
\usetikzlibrary{trees, snakes}
\usepackage[all]{xy}
\usepackage{mathtools}

\usepackage{amsmath, amsthm}  
\usepackage{amsfonts, amssymb}
\usepackage{stmaryrd}         
\usepackage{bm}               

\usepackage[ breaklinks,  
            pagebackref,           
            pdfborder={0 0 0}
            ]{hyperref}	 

 \let\footnote=\endnote
\ifx\thesection\undefined
\newtheorem{thm}{Theorem}
\else
\newtheorem{thm}{Theorem}[section]
\fi
 \theoremstyle{definition}
 \newtheorem{exple}[thm]{Example}

\theoremstyle{definition}
	\newtheorem{defi}[thm]{Definition}

 \theoremstyle{plain}
 \newtheorem{prop}[thm]{Proposition}
 \newtheorem{lem}[thm]{Lemma}
 \newtheorem{cor}[thm]{Corollary}

\theoremstyle{remark}
 \newtheorem{remark}[thm]{Remark}
	\newtheorem{rque}[thm]{Remark}	
	
\usepackage{xargs}

\newcommand{\com}{\operatorname{Com}}

\newcommand{\perm}{\operatorname{Perm}}

\newcommand{\id}{\operatorname{id}}
\newcommand{\ME}{\operatorname{ME}}

\usepackage{endnotes}
\usepackage{stmaryrd}
\usepackage{textcomp}
\usepackage{framed}

\newcommand{\Bc}{\mathcal{B}}

\newcommand{\Nc}{\mathcal{N}}
\newcommand{\Pc}{\mathcal{P}}
\newcommand{\tPc}{\widetilde{\mathcal{P}}}
\newcommand{\Qc}{\mathcal{Q}}
\newcommand{\Tc}{\mathcal{T}}

\newcommand{\Kb}{\mathbb{K}}
\newcommand{\Nb}{\mathbb{N}}
\newcommand{\Sb}{\mathbb{S}}
\newcommand{\Zb}{\mathbb{Z}}

\newcommand{\Set}{\textsf{Set}}
\newcommand{\Vect}{\textsf{Vect}}
\newcommand{\Bij}{\textsf{Bij}}
\newcommand{\Ord}{\textsf{Ord}}
\newcommand{\Op}{\textsf{Op}}
\newcommand{\OpSh}{\textsf{OpSh}}

\usepackage{todonotes}

\title[Operads with compatible CL-shellable partition posets are PBW]{Operads with compatible CL-shellable partition posets admit a Poincaré--Birkhoff--Witt basis}

\author[J. Bellier-Mill\`{e}s]{Joan Bellier-Mill\`{e}s}
\address{B.-M.: IMT ; UMR5219\\ Université de Toulouse ; CNRS ; UPS, F-31400 Toulouse, France}
\author[B. Delcroix-Oger]{B\'{e}r\'{e}nice Delcroix-Oger}
\address{D.-O.: IRIF, UMR 8243, Université Paris 7, Sorbonne Paris Cit\'e}
\author[E. Hoffbeck]{Eric Hoffbeck}
\address{H.: Universit\'e Paris 13, Sorbonne Paris Cit\'e, LAGA, CNRS, UMR 7539, 99 avenue Jean-Baptiste Cl\'ement, F-93430, Villetaneuse, France}

\thanks{This work was partially supported by the LabEx CIMI, the grants ANR-13-BS02-0005-02 project CATHRE, ANR-14-CE25-0008-01 project SAT, ANR-16-CE40-0003 project ChroK}
 \date{}


\begin{document}

\maketitle

\begin{abstract}
In 2007, Vallette built a bridge across posets and operads by proving that an operad is Koszul if and only if the associated partition posets are Cohen-Macaulay. Both notions of being Koszul and being Cohen--Macaulay admit different refinements: our goal here is to link two of these refinements. We more precisely prove that any (basic-set) operad whose associated posets admit isomorphism-compatible CL-shellings admits a Poincaré--Birkhoff--Witt basis. Furthermore, we give counter-examples to the converse.
\end{abstract}

\tableofcontents

\section*{Introduction}

The family of partitions of the sets $\{1, \ldots, n \}$ are partially ordered by refinement. For example, we have $\{ 1,\, 3\} \{ 2\} \{4\} \leq \{ 1,\, 3\} \{ 2,\, 4\}$. The smallest partition is given by the collection $\{1\} \cdots \{n \}$ and the largest partition is given by the single set $\{1, \ldots, n \}$.  This family of partially ordered sets (posets for short) can be built from the operad $\com$ encoding commutative algebras, and the homology of the simplicial set associated with this family of partition posets is isomorphic to the homology associated with the operad $\com$ (see \cite{bF04} and references within). 
More generally, Vallette shows in \cite{bV07} that a family of partition type posets can be associated with any (basic-set) operad in the category of sets, and he proves the equivalence between two homological properties: the Koszul property of the operad and the Cohen-Macaulay property of its associated partition posets. \\

The goal of this article is to show how to relate two refinements of these properties.
On the one hand, the Koszul property is usually hard to prove, and one useful criterion is the existence of a Poincaré--Birkhoff--Witt (PBW for short) basis, which is a basis of the operad satisfying some combinatorial properties. On the other hand, the Cohen--Macaulay property of a poset is implied by the shellability of this poset. And the shellability itself has many refinements, for instance EL-shellability, CL-shellability (which is equivalent to the existence of a recursive atom ordering) or the totally semi-modular property (see \cite{Wachs} for a general survey).\\

A natural question to ask is therefore to find a link between combinatorial conditions living in two very different worlds.
Unlike the Koszul property which is purely homological, a PBW basis is a combinatorial data subject to a compatibility condition relative to the operadic composition.
The different refinements of shellability are given by conditions on each poset of the family defined by Vallette in \cite{bV07}. 
In order to get a strong link between one of these notions with the existence of a PBW basis, we need to add a compatibility hypothesis.\\

Our main theorem states that if the posets associated with a basic-set operad are CL-shellable and satisfy a compatibility hypothesis, then the operad has a PBW basis. This provides a new method for proving that an operad admits a PBW basis, and therefore to show that an operad is Koszul. As in the work \cite{bV07}, these results are valid over any field $k$ or over $\Zb$.

Moreover, let us remark that we work in the context of weight graded operads, and it applies to weight graded associative algebras. It was not the case of the results of \cite{bV07} which apply to connected operads. We obtain the new notion of $A$-partitions associated with an associative algebra $A$ and a way to get a PBW basis by means of a CL-labelling satisfying the previous compatibility hypothesis.\\
 
The converse direction is more complicated. 
First we give the example of an algebra admitting a lexicographic PBW basis and whose associated posets are not isomorphism compatible CL-shellable.
Then we give another stronger example of an algebra such that any (lexicographic or not) PBW basis cannot be linked with an isomorphism compatible CL-shelling. 
The shellability of a poset is a weaker notion than the CL-shellability. It is therefore natural to ask whether the existence of a PBW basis could be related to the shellability of the poset associated to a basic-set operad with a compatibility hypothesis. 
Further work is needed to answer this question.

\subsection*{Outline of the paper}
In the first part of the paper, we make some recollections on operads. Then we describe (following \cite{bV07}) the poset associated to a shuffle operad and the compatibility condition we will require later on. We finally state and prove our main theorem and give a counter-example for the converse direction.

\subsection*{Acknowledgements}
The authors would like to thank Bruno Vallette for his careful reading of a first version of this article and his numerous comments. 
Moreover, we'd like to thank Rafael González, Josh Hallam and Yeison Quiceno for pointing out an error concerning a potential EL-labelling on the poset associated with the operad $\perm$ in an earlier version of the article. We have removed the problematic part from this version.

\subsection*{Notation}
Let $k$ be the ring $\Zb$ or any field of characteristic $0$. All $k$-modules are assumed to be projective. 
We mainly work with the categories $\Set$ (whose objects are sets and morphisms are all maps between them) 
and $\Vect$ (whose objects are $k$-vector spaces and morphisms are all $k$-linear maps between them).
We denote by $\Bij$ the category of non-empty finite sets (with bijections between them as morphisms) and
by $\Ord$ the category of non-empty finite ordered sets (with order-preserving bijections between them as morphisms).

\section{Operads and associated constructions}

This section is devoted to operads and related constructions. After some recollections on operads and shuffle operads, we present the link between PBW basis and the Koszulness of an operad. Finally, we give definitions related to the normalised bar construction.

\subsection{Recollections on operads}

When talking about operads, we follow essentially the notations of the book of Loday and Vallette \cite{LV12}. We are interested in operads in $\Set$, usually denoted by $\tPc$, and their linearised version $\Pc = k \langle \tPc \rangle$ in $\Vect$. The operads we consider are symmetric operads, but we will also need their shuffle versions, introduced by Dotsenko and Khoroshkin in \cite{DK}, following ideas of \cite{eH10}.

\subsubsection{Operads and shuffle operads}

Recall that an $\Sb$-module $\Pc=\{\Pc(n)\}_{n \in \Nb}$ in $\Vect$ is a collection of $k$-vector spaces $\Pc(n)$ endowed with an action of the symmetric group $\Sb_n$. On $\Sb$-modules, we define the monoidal product
\[
(\Pc \circ \Qc)(n)=\bigoplus_{k \geq 0} \Pc(k)\otimes_{\Sb_k} \left(\bigoplus_{i_1 + \cdots + I_k = n}\mathrm{Ind}_{\Sb_{i_1} \times \cdots \times \Sb_{i_k}}^{\Sb_n}\left(\Qc(i_1) \otimes \cdots \otimes \Qc(i_k)\right)\right).
\]
 A \emph{symmetric operad} is an $\Sb$-module $\Pc=\{\Pc(n)\}_{n \in \Nb}$ endowed with a composition map and a unit map which make it into a monoid in the category of $\Sb$-modules.

\begin{remark}
There is an analog definition for operads in the category $\Set$.
\end{remark}

We will also sometimes consider more generally an operad as a contravariant functor $\Pc$ from $\Bij$ to $\Vect$ satisfying similar axioms; this allows us to consider $\Pc(I)$ for $I$ any finite set, 
and we then recover $\Pc(n)$ of the first definition as $\Pc(\{1, \ldots, n \})$.

\begin{defi}
The category of $\mathbb{N}$-graded modules consists of contravariant functors from $\Ord$ to $\Vect$. When we restrict to the sets $\{ 1,\, \ldots ,\, n\}$, it coincides with the category of collections $\{ \Pc(n)\}_{n  \in \Nb}$, where each $\Pc(n)$ is a $k$-vector space. It can be equipped with a monoidal product called shuffle product and defined by
\[
(\Pc \circ_{sh} \Qc)(I)=\bigoplus_k \Pc(k)\otimes \left( \bigoplus_{f:I\to\{1, \ldots, k\}}\Qc(f^{-1}(1)) \otimes \cdots \otimes \Qc(f^{-1}(k))\right),
\]
where $f$ ranges over shuffle surjections, that is surjections such that we have $\min f^{-1} (i) < \min f^{-1} (j)$ whenever $i<j$.
A \emph{shuffle operad} can be then defined similarly as a $\mathbb{N}$-graded module $\Pc=\{\Pc(n)\}_{n \geq 0}$ with a composition map and a unit map which make it into a monoid in this monoidal category of $\mathbb{N}$-graded modules.
\end{defi}

In particular, given two elements $p$ and $q$ in a shuffle operad, 
the possible partial compositions are not just given by the $\circ_i$'s as for symmetric operads but by the $\circ_{i,w}$'s where $w$ is a shuffle of the entries of $p$ and $q$ 
such that the first element of $q$ comes immediately after the $(i-1)$-th element of $p$. Such a $w$ is then called a \emph{pointed shuffle}.

\begin{defi}
We denote by $\Op$ the category of symmetric operads (in $\Set$ or $\Vect$) and by $\OpSh$ the category of shuffle operads.
\end{defi}

There is a forgetful functor from $\Op$ to $\OpSh$, which we denote by $(-)_{sh}$. This functor forgets the action of the symmetric groups but it retains the whole structure of a symmetric operad. 
Recall that a symmetric operad $\Pc$ and its associated shuffle operad $\Pc_{sh}$ have the same underlying $k$-vector space.\\

We denote by $\Tc (E)$ the free operad in $\Set$ (respectively in $\Vect$) on the $\Sb$-module $E$ in $\Set$ (respectively in $\Vect$). We use the same notation in both categories, the category where $E$ sits will make clear where $\Tc (E)$ sits.

As it is explained for example in Appendix A of \cite{bF17} (see also Section 5.6 in \cite{LV12}), for an $\Sb$-module $E$ in $\Set$, the component of arity $n$ of the free operad $\Tc (E)$ 
can be identified with the set of equivalence classes of non-planar rooted trees, where vertices are labelled by elements of $E$ and leaves are labelled by the integers between $1$ and $n$
(the equivalence relation is equalizing the action of the symmetric group on $E$ and the induced action on the tree). The set $\Tc(E)$ is acted upon by the group $\Sb_n$.
For an $\Sb$-module $E$ in $\Vect$, a similar construction is possible, and labelled trees can be thought of tensor products of elements of $E$ arranged on a tree, sometimes called treewise tensors.\\

Similarly, the elements of the free shuffle operad $\Tc_{sh}(E)$ are given by rooted planar trees, with vertices labelled by elements of $E$ and leaves labelled by the integers between $1$ and the arity of the tree, in such a way that above each vertex, the inputs are in increasing order from left to right (if the input is a vertex, the associated integer is the smallest label of the leaves above it). We refer to Section 8.2 in \cite{LV12} for more details.

\subsubsection{Linear basis and weight grading}

For a $k$-linear basis $\Bc^E$ of an $\Sb$-module $E$ in $\Vect$, a basis denoted by $\Bc^{\Tc(E)}$ of $\Tc(E)$ is given by all shuffle trees labelled by elements of $\Bc^E$ (see \cite[3.1]{eH10}).
The operad $\Tc(E)_{sh}$ has the same basis, which is stable under shuffle composition.\\

When $E$ is an $\Sb$-module and $R \subset \Tc(E)$, we consider the quotient operad $\Tc (E)/(R)$ of the free operad on $E$ by the ideal generated by the set of relations $R$. 

The free operad is naturally endowed with a \emph{weight grading} counting the number of vertices in the tree representation.
When the $\Sb$-module $R$ is homogeneous for this weight grading, we obtain that the quotient operad $\Tc(E)/(R)$ is also weight graded. We assume that we are in this situation. Moreover, we suppose that $E$ is minimal in the sense that there exists no relation of weight $1$ in $R$.\\

We are mainly interested in the two following situations:
\begin{itemize}
\item
\emph{Connected} operads: $E(n)=0$ for $n=0$ and $n=1$.
\item
Associative algebras: $E(n)=0$ for $n \neq 1$.
\end{itemize}

In this article, we restrict ourselves to finite sets (or finite-dimensional spaces) of generators $E$. 

In the case of associative algebras, we denote by $T(E)$ the free associative algebra on $E$, which is also the tensor algebra on $E$. Elements in weight $d$ are given by $T(E)^{(d)} = E^{\otimes d}$.

\subsection{PBW bases and Koszulness}

Let $\Bc^E$  be a $k$-linear basis of an $\Sb$-module $E$ (in $\Vect$) and $\Bc^{\Tc(E)}$ is the associated (monomial) basis of $\Tc(E)$.

In the following, we assume that $\Bc^E$ is partially ordered, in a way compatible with the arity:
\begin{center}
$\mu < \nu$ if $\mu \in \Bc^E(k)$ and $\nu \in \Bc^E(l)$ with $k < l$.
\end{center}

Let us extend this order to $\Bc^{\Tc(E)}$ in a way compatible with the shuffle composition, that is:
for $\alpha,\, \alpha' \in \Bc^{\Tc(E)}(m)$ and $\beta,\, \beta' \in \Bc^{\Tc(E)}(n)$,
\begin{align}
\left\{ \begin{array}{l} \alpha \leq \alpha'\\ \beta \leq \beta' \end{array} \Rightarrow \forall i,\, \forall w \text{ pointed shuffle},\ \alpha \circ_{i,w} \beta \leq \alpha' \circ_{i,w} \beta'.\right.
\end{align}

Examples of suitable orders can be found in \cite{eH10} or \cite{DK}. 

\begin{defi}\label{defPBW}
A \emph{Poincaré-Birkhoff-Witt basis} (PBW basis for short) for $\Pc = \Tc(E)/(R)$ is a subset $\Bc^{\Pc}$ of $\Bc^{\Tc(E)}$ such that
\begin{itemize}
\item $1 \in \Bc^{\Pc}$,
\item $\Bc^E \subset \Bc^{\Pc}$,
\item $\Bc^{\Pc}$ represents a basis of the $\Kb$-module $\Pc$,
\end{itemize}
and satisfying the following conditions:
\begin{enumerate}
\item for $\alpha,\, \beta \in \Bc^{\Pc}$ and $w$ a pointed shuffle of the composition $\alpha \circ_i \beta$, either $ \alpha \circ_{i,w} \beta \in \Bc^{\Pc}$, or $\alpha \circ_{i,w} \beta = \sum_\gamma c_{\gamma} \gamma$, where the $\gamma$'s are in $\Bc^{\Pc}$ and satisfy  $\gamma < \alpha \circ_{i,w} \beta$ in the ordered basis of $\Tc(E)$; 

\item a treewise tensor $\alpha$ based on a tree $\tau$ 
is in $\Bc^{\Pc}$ if and only if for every internal edge $e$ of $\tau$, the restricted treewise tensor $\alpha_{|\tau_e}$ is in $\Bc^{\Pc}$. (The subtree $\tau_e$ of $\tau$ is the maximal subtree of $\tau$ having only the internal edge $e$.)
\end{enumerate}
\end{defi}

\begin{remark}
Note that in this definition, we do not need the order on $\Bc^E$ nor the order on $\Bc^{\Tc(E)}$ to be total.
Intuitively, it is enough to be able to write generating relations of $\Pc=\Tc(E)/(R)$ as one element being equal to a linear combination of smaller ones.
For instance it is the case when the relations in the presentation of the associated shuffle operad is given by a convergent rewriting system in the monomial basis of $\Tc(E)$
(which makes sense not only in the set-theoretical context but also in the linear context, see \cite{GHM} Section 3).
The normal forms for such a rewriting system are the minimal elements for the associated order and provide us with a PBW basis of the operad.
In such a case, the elements in $\Bc^E$ do not need to be comparable.
\end{remark}

\begin{thm}
An operad equipped with a partially ordered PBW basis is Koszul.
\end{thm}

\begin{proof}
The proof of the usual criterion \cite[Theorem 2.5.1]{eH10} for a totally ordered PBW basis can still be applied when the PBW basis is only partially ordered.
Indeed, any partial order can be extended to a total order. Choosing such a total order, the conditions of a PBW basis will then remain true for this extended order. They allow us to construct a filtration of subcomplexes, even if the total order might not be compatible with pointed shuffle compositions.
The total order is only needed so that the stages of the filtration remain as in the usual case and so that the spectral argument can be applied.
\end{proof}

\subsection{Normalised reduced bar construction}\label{normalisedBarConstruction}

Let $\Pc$ be an operad. 
We denote by $\Nc(\Pc)$ the \emph{normalised bar construction} of $\Pc$ \cite[Section 4]{bF04} or \cite[3.1.2]{LV12},
which can be defined as the nerve of the simplicial bar construction. 
Following these references, we denote by $\Nc_l(\Pc)$ the $\Sb$-module represented by non-degenerate $l$-levelled trees. A representing element in the normalised bar construction writes $(\nu_1^1;\, \nu_1^2,\, \ldots ,\, \nu_{k_1}^2 ;\, \ldots ;\, \nu_1^l,\, \ldots,\, \nu_{k_l}^l)$, where each tuple $(\nu_1^j,\, \ldots ,\, \nu_{k_j}^j)$ is not only made up of identities. An example of a non-degenerate $3$-levelled tree is represented in Figure \ref{fig-NPel}.
\begin{figure}[h]
\[\xymatrix@R=10pt@C=10pt{ & 3\ar[d] &
1\ar[dr] & 4\ar[d] & 5\ar[dl] & 2\ar[dr] & & 6\ar[dl] \\
3\ar@{.}[r] & *+[F-,]{\nu_1^3}\ar@{.}[rr]\ar[d] & &*+[F-,]{\nu_2^3}\ar@{.}[rr]\ar[dr]\ar@{.}[rrr] & & & *+[F-,]{\nu_3^3}\ar[dll]\ar@{.}[r] &
 & \\
2\ar@{.}[r] & *+[F-,]{\nu_1^2}\ar@{.}[rrr]\ar[drr] & & & *+[F-,]{\nu_2^2}\ar@{.}[rrr]\ar[dl] & & & \\
1\ar@{.}[rrr] & & & *+[F-,]{\nu_1^1}\ar@{.}[rrrr]\ar[d] & & & & \\
& & & & & & & & \\ }\]
\caption{Non-degenerate $3$-levelled tree}\label{fig-NPel}\end{figure}
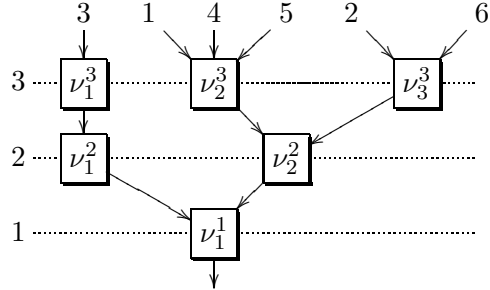

When the operad is weight-graded, we get that the normalised bar construction also is. Moreover, we assume that the operad $\Pc$ is presented by $\Tc(E)/(R)$ and is weight-graded by $\Pc^{(1)} = E$. In this case, we define the sub-$\Sb$-module $\Nc(E)$ of $\Nc(\Pc)$ by
\[
\Nc(E) \coloneqq \bigoplus_{d \geq 0} \Nc_{d}(\Pc)^{(d)}.
\]
Elements in $\Nc(E)$ can be described by levelled trees having exactly one generator different from the identity at each level. We write $\nu^1 \circ_{i_1} \nu^2 \circ_{i_2} \cdots \circ_{i_{l-1}} \nu^l \coloneqq (\cdots (\nu^1 \circ_{i_1} \nu^2) \circ_{i_2} \cdots ) \circ_{i_{l-1}} \nu^l$ for such an element. The notation $\nu \circ_i \mu$ means here that $\mu$ is put at the $i$th entry of $\nu$ just one level above.
For example, in the two cases we are interested in, we get the following descriptions.
\begin{itemize}
\item
When $\Pc$ is a connected binary operad, that is $E(n) = \{0\}$ for $n \neq 2$, we get that
\[
\Nc(E) \coloneqq I \oplus \bigoplus_{d \geq 1} \Nc_{d}(\Pc)(d+1),
\]
where $\Nc(\Pc)(d)$ is the $\Sb_d$-module of arity $d$ elements in $\Nc(\Pc)$.
\item
When $E$ is concentrated in arity $1$ (that is, when $\Pc$ is an associative algebra), $\Nc(\Pc)$ is nothing but the usual normalised bar construction of an associative algebra and we get $\Nc (E) \cong T(E)$.
\end{itemize}

\section{Posets}

We recall in this section some definitions on posets and on the CL-shellability of a poset. We also recall how to associate a partition type poset with an operad and provide examples. We finally present the definition of CL-labellings compatible with isomorphisms of subposets. 

\subsection{Shellability of a poset}

We recall in this subsection basic definitions of poset topology, following \cite{Wachs}.

\begin{defi}
A \emph{poset} $(\Pi,\, \leq)$ is a set $\Pi$ equipped with a partial order relation $\leq$.
\end{defi}

\begin{defi}
A poset $\Pi$ is said to be \emph{bounded} if it has a top element $\hat 1$ and a bottom element $\hat 0$.
\end{defi}

For elements $x < y$ in a poset $\Pi$, if there exists no $z$ such that $x < z < y$, then we say that $y$ \emph{covers} $x$. The covering relation is denoted by $x \prec y$.

\begin{defi} 
A \emph{maximal chain} is a chain which is not strictly contained in another chain. It is thus a chain $x_0 \prec \cdots \prec x_m$, with $x_0$ (resp. $x_m$) a minimal (resp. maximal) element in the poset. 
The \emph{length} of a chain $x_0<\cdots < x_m$ is $m$.
A poset $\Pi$ is \emph{pure} if all its maximal chains have the same length, which is then the \emph{height} $l(\Pi)$ of the poset. A pure and bounded poset is \emph{graded}.
\end{defi}

For a bounded poset $\Pi$, we consider the following definitions.

\begin{defi}
A pure poset is \emph{Cohen--Macaulay} if all its reduced homology groups, but the top ones, vanish.
\end{defi}

Different notions refine the notion of Cohen--Macaulayness. One of them is the notion of CL-shellability.

\begin{defi}
To every maximal chain $r$ of $[\hat 0,\, x]$ and to every interval $[x,\, y]$, we can associate the \emph{closed rooted interval}
\[ [x,\, y]_r \coloneqq \{ z\in r\} \cup \{ z \in [x,\, y]\}.\]
\end{defi}

We can now define the notion of \emph{chain-edge labelling}. The reason for using a poset $\Lambda$ to label chains will appear clearly in the definition of CL-labelling.

\begin{defi}
We denote by $\ME(\Pi)$ the set of pairs $(c,\, x \prec y)$ consisting of a maximal chain $c$ of $\Pi$ and an edge $x \prec y$ along that chain. A \emph{chain-edge labelling} of $\Pi$ is a map $\lambda : \ME(\Pi) \to \Lambda$, where $\Lambda$ is some poset, satisfying: if two maximal chains coincide along their bottom $d$ edges (for some integer $d$), then their labels also coincide along these edges.
\end{defi}

\begin{remark}\label{remark: restricted CL-labelling}
Let $[x,\, y]_r$ be a rooted interval of $\Pi$ and $\lambda : \ME(\Pi) \to \Lambda$ be a chain-edge labelling of $\Pi$. The map $\lambda$ induces a chain-edge labelling $\lambda_r^{x,\, y} : \ME([x,\, y]) \to \Lambda$. Moreover, we can associate with a maximal chain $c = (x_0 \prec \cdots \prec x_m)$ in $[x,\, y]$ a tuple of symbols in $\Lambda$ given by
\[
\left(\lambda_{x,\, y}^r ((c,\, x_0 \prec x_1)),\, \ldots,\, \lambda_{x,\, y}^r ((c,\, x_{m-1} \prec x_m))\right).
\]
\end{remark}

\begin{defi}
Let $\Pi$ be a poset. A \emph{chain-lexicographic labelling} (CL-labelling, for short) of $\Pi$ is a chain-edge labelling $\lambda : \ME(\Pi) \to \Lambda$ such that in each closed rooted interval $[x,\, y]_r$ of $\Pi$, there is a unique maximal chain whose associated labels in $\Lambda$ by $\lambda_r^{x,\, y}$ forms a strictly increasing chain in $\Lambda$. Moreover, this strictly increasing chain, seen as a tuple, precedes lexicographically all the tuples associated with other maximal chains of $[x,\, y]_r$. A poset that admits a CL-labelling is said to be \emph{CL-shellable}.
\end{defi}

It is a well-known result that any pure CL-shellable poset is Cohen--Macaulay, but the converse is not true.

\subsection{Poset associated with an operad}

Let $\tPc$ be a set operad. In \cite{bV07}, Vallette defines the notion of $\tPc$-partitions and he puts a partial order on the set of $\tPc$-partitions. He obtains what he calls the \emph{operadic partition poset} $\Pi_{\tPc}$. In order to recover a well-known order on some examples, we consider in this article the reverse order as the one given in \cite{bV07}.

We recall the definitions from \cite[3.1]{bV07}. Let $I \subset [n] \coloneqq \{ 1,\, \ldots ,\, n\}$ be a subset. 

\begin{defi}
A \emph{$\tPc$-partition} of $I$ is a collection $\lambda = \{ B_1,\, \ldots ,\, B_r\}$ with $B_k$ belongs to $\tPc(I_k)$ for $\{ I_1,\, \ldots ,\, I_r\}$ a partition of $I$. 
\end{defi}
An element $B \in \tPc(I)$ writes $\overline{\nu \times (x_1,\, \ldots ,\, x_n)}$. It is equal to the orbit of an element $\nu \times (x_1,\, \ldots ,\, x_n) \in \tPc_n \times \mathcal{I}$, where $\mathcal{I}$ is the set of ordered sequences of elements of $I$, each element appearing once, under the diagonal action of the symmetric group $\Sb_n$. 

We recall the partial order defined on $\tPc$-partitions:
\begin{defi}
The partial order on the $\tPc$-partitions of $I$ is given by
\begin{equation*}
 \lambda = \{B_1,\, \ldots ,\, B_r\} \leq \omega = \{C_1,\, \ldots ,\, C_s\},
\end{equation*}
where $B_k$ belongs to $\tPc(I_k)$ and $C_l$ to $\tPc(J_l)$ if for any $l\in \{1,\, \ldots ,\, s\}$, there exists a tuple $\{p_1^l,\, \ldots ,\, p_{s(l)}^l\} \subset \{1,\, \ldots ,\, r\}$ such that $\left\{I_{p_1^l},\, \ldots ,\, I_{p_{s(l)}^l}\right\}$ is a partition of $J_l$ and if there exists an element $\nu_l$ in $\tPc_{s(l)}$ such that
\[
C_l = \gamma \left(\nu_l \times \left(B_{p_1^l},\, \ldots ,\, B_{p_{s(l)}^l}\right)\right).
\]
\end{defi}

\begin{defi}
We say that the set operad $\tPc$ is \emph{basic-set} when for all $\nu_1$, \ldots ,\, $\nu_k$ in $\tilde \Pc$, the composition map $\tPc \to \tPc$, $\nu \mapsto \gamma (\nu , \nu_1 , \cdots , \nu_k)$ is injective.
\end{defi}

From now on, we assume that the operad $\tPc$ is basic-set. This condition is necessary to define \emph{isomorphisms of subposets} and to prove Theorem \ref{Theorem: order complex and normalised bar construction}. For these reasons, we need this assumption in Section \ref{Section: main theorem}.\\

As the considered operads are basic-set, the following proposition holds:
\begin{prop}
Let $\lambda$ and $\omega$ be $\tPc$-partitions as in the definition above, with $\lambda \leq \omega$.
For any $l \in \{1, \ldots, s\}$, the element $\nu_l$ is uniquely determined once the order of the $B_{p^l_k}$'s is fixed. 
\end{prop}

It is now possible to consider a $\tPc$-partition measuring the difference between $\lambda$ and $\omega$, by removing the $\nu_l$'s which are identities. 
More precisely, we define the $\tPc$-partition $\delta_{\lambda}^{\omega}$ of the subset
\[
D_{\lambda}^{\omega} = \bigcup_{\substack{1 \leq l \leq s\\ \nu_l \neq \id}} \left\{ \min I_{p_1^{l}},\, \ldots ,\, \min I_{p_{s(l)}^l} \right\}
\]
of cardinal $r'\leq r$ by
\[
\delta_{\lambda}^{\omega} \coloneqq \{A_1,\, \ldots ,\, A_{s'}\} = \bigcup_{\substack{1 \leq l \leq s\\ \nu_l \neq \id}}\left\{ \overline{\nu_l \times (\min I_{p_1^{l}},\, \ldots ,\, \min I_{p_{t(l)}^l})}\right\}.
\] 

\begin{exple}
\begin{enumerate}
\item
When $\lambda = (B_1,\, \ldots ,\, B_r) \prec \omega = (C_1,\, \ldots ,\, C_s)$ is a covering relation, we get that the $\tPc$-partition $\delta_{\lambda}^{\omega}$ is $\{ A\}$ where $A = \overline{\nu \times (x_1,\, \ldots ,\, x_t)}$, for $x_j  = \min I_{k_j}$ for some $k_j$ and for $\nu \in \tPc(t)$ which cannot be written as a product of two non-trivial elements in $\tPc$.
\item
In the case where $\tPc$ is an associative algebra, the $\tPc$-partition $\delta_{\lambda}^{\omega}$ is always a set of cardinal $1$.
\end{enumerate}
\end{exple}

Again, we assume that the set operad $\tPc$ is presented by $\Tc(E)/(R)$ and is endowed with the weight-grading given by $\tPc^{(1)} = E$.

The following proposition follows from the fact that $\tPc^{(0)} = I$:

\begin{prop}
The partial order defined in \cite{bV07} is an anti-symmetric partial order. 
\end{prop}

An element $\nu$ in $\tPc$ induces canonically an element $\bar{\nu}\coloneqq\overline{\nu \times (1, ….n)}$ in $\Pi_{\tPc}$.

\begin{defi}
We define the subposets $\Pi_{\tPc}^{(d)}$, generated by elements in $\tPc$ of weight $d$, as follows:
\[
\Pi_{\tPc}^{(d)} \coloneqq \{ \lambda \in \Pi_{\tPc}\ ; \ \exists \nu \in \tPc^{(d)} \text{ such that } \lambda \leq \bar{\nu}\}.
\]
\end{defi}

\begin{remark}
When the operad $\tPc$ admits a binary quadratic presentation, the subposet $\Pi_{\tPc}^{(d)}$ coincides with what is called $\Pi_{\tPc}(d+1)$ in \cite{bV07}.
\end{remark}

Using the fact that the operad $\tPc$ is homogeneous, we get that the posets $\Pi_{\tPc}^{(d)}$ are pure with minimal elements of the form $\{ \overline{\id \times (1)},\, \ldots ,\, \overline{\id \times (k)}\}$ and maximal elements given by elements in $\tPc^{(d)}$.\\

 We can associate an analog partition poset to a shuffle operad in such a way that $\Pi_{\tPc} \cong \Pi_{\tPc_{sh}}$, where $(-)_{sh}$ is the forgetful functor from symmetric operads to shuffle operads. An element in $\Pi_{\tPc_{sh}}$ is a $\tPc_{sh}$-partition of $[n]$, that is a tuple $( B_1,\, \ldots ,\, B_r,\, w)$, such that, for all $j$, $B_j \in \tPc_{sh}(I_j)$, where $I_j = \{ i_1+\cdots +i_{j-1}+1,\, \ldots ,\, i_1+ \cdots + i_j\}$ for some $i_j$'s such that $i_1 + \cdots +i_r = n$, and $w$ is a $(i_1,\, \ldots ,\, i_r)$-unshuffle in $\Sb_n$ satisfying $w^{-1}(i_1+\cdots +i_j+1) < w^{-1}(i_1+\cdots +i_{j'}+1)$ when $0 \leq j < j' < r$. We have moreover $\max \Pi_{\tPc_{sh}}^{(d)} \cong \tPc^{(d)}$. In what follows, we will use the notation $\Pi_{\tPc}$ to denote the partition poset $\Pi_{\tPc} \cong \Pi_{\tPc_{sh}}$.\\

To every poset $\Pi$, we can associate an abstract simplicial complex $\Delta(\Pi)$ called the \emph{order complex} of $\Pi$. The vertices of $\Delta(\Pi)$ are the elements of $\Pi$ and the faces of $\Delta(\Pi)$ are the chains of $\Pi$. Theorem 7 of \cite{bV07} extends to the weight graded setting as follows.

\begin{thm}\label{Theorem: order complex and normalised bar construction}
Let $\tPc$ be a basic-set operad admitting a homogeneous presentation $\tPc = \Tc\left( E\right)/(R)$. We endow $\tPc$ with the  weight grading such that $\tPc^{(1)} = E$. Then, for any $d \in \Nb$, we have an isomorphism of presimplicial $\Sb$-modules between the order complex $k \left[\Delta_* \left(\Pi_{\tPc}^{(d)}\right) \right]$ and the normalised bar construction $\Nc_*(\Pc)^{(d)}$, where $\Pc$ is the linear operad associated with $\tPc$.
\end{thm}

\begin{proof}
The proof presented in \cite{bV07} works in our setting. It is enough to remark that the weight is preserved by the bijection.
\end{proof}

\begin{rque}\label{maximalChains}
When restricted to maximal chains, this isomorphism identifies maximal chains of $\Pi_{\tPc}^{(d)}$ with $\Nc\left( E\right)^{(d)}$ (see Section \ref{normalisedBarConstruction} for a definition of $\Nc\left( E\right)$). In the following, we will only use this isomorphism of $\Sb$-modules between maximals chains in $\Pi_{\Pc}^{(d)}$ and $\Nc\left( E\right)^{(d)}$.
\end{rque}

\subsection{Main examples}

We present here some basic examples.

\subsubsection{Algebras}

We consider an algebra on two generators $a$ and $b$ given by the presentation
\[
A \coloneqq T(a,\, b) / (ab = ba ).
\]
This algebra is basic-set and the presentation is quadratic (hence homogeneous). The associated partition posets $\Pi_A^{(d)}$, whose maximal elements are represented by elements in $T(a,\, b)$ made up of tensor products of $d$ elements in $\{ a,\, b\}$, are pure for all $d\geq 0$.

The posets $\Pi_A^{(d)}$ for $d = 1$ and $2$ are the following ones:\\
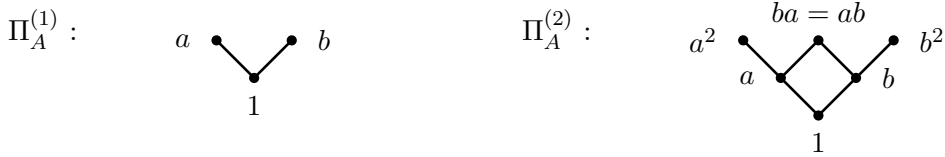
\begin{minipage}[t]{.46\linewidth}
$\Pi_A^{(1)}$ :\hspace{1cm}
\begin{tikzpicture}[line width=1pt, scale=0.5, baseline=(a.center)]
\coordinate (a) at(0,1);
\coordinate (b) at(1,0);
\draw (a)--(b);
\coordinate (c) at(2,1);
\draw (c)--(b);
\node [left =0.2cm of a] {$a$};
\node [below =0.1cm of b] {$1$};
\node [right =0.2cm of c] {$b$};
\foreach \x in {a,b, c}
\draw[fill=black] (\x) circle(0.1);
\end{tikzpicture}
\end{minipage} \hfill
\begin{minipage}[t]{.46\linewidth}
$\Pi_A^{(2)}$ :\hspace{1cm}
\begin{tikzpicture}[line width=1pt, scale=0.5, baseline=10mm]
\coordinate (a) at(2,0);
\coordinate (b) at(1,1);
\coordinate (c) at(3,1);
\coordinate (d) at(0,2);
\coordinate (e) at(2,2);
\coordinate (f) at(4,2);
\draw (a)--(d);
\draw (b)--(e);
\draw (c)--(e);
\draw (a)--(f);
\node [below =0.1cm of a] {$1$};
\node [left =0.2cm of b] {$a$};
\node [right =0.2cm of c] {$b$};
\node [left =0.2cm of d] {$a^2$};
\node [above =0.1cm of e] {$ba = ab$};
\node [right =0.2cm of f] {$b^2$};
\foreach \x in {a,b,c,d,e,f}
\draw[fill=black] (\x) circle(0.1);
\end{tikzpicture}
\end{minipage}

As a second example, we consider the algebra
\[
B \coloneqq T(x,\, y) / (xx-xy,\, yx-yy).
\]
This algebra is basic-set and the presentation is quadratic. For every $d\geq 0$, the poset $\Pi_B^{(d)}$ is pure. 
The posets $\Pi_B^{(d)}$ for $d = 1$ and $2$ are the following ones:\\
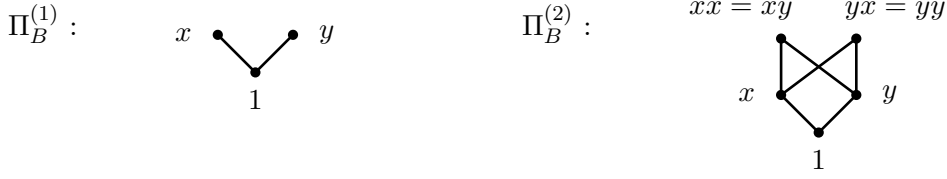
\begin{minipage}[t]{.46\linewidth}
$\Pi_B^{(1)}$ :\hspace{1cm}
\begin{tikzpicture}[line width=1pt, scale=0.5, baseline=(a.center)]
\coordinate (a) at(0,1);
\coordinate (b) at(1,0);
\draw (a)--(b);
\coordinate (c) at(2,1);
\draw (c)--(b);
\node [left =0.2cm of a] {$x$};
\node [below =0.1cm of b] {$1$};
\node [right =0.2cm of c] {$y$};
\foreach \x in {a,b, c}
\draw[fill=black] (\x) circle(0.1);
\end{tikzpicture}
\end{minipage} \hfill
\begin{minipage}[t]{.46\linewidth}
$\Pi_B^{(2)}$ :\hspace{1cm}
\begin{tikzpicture}[line width=1pt, scale=0.5, baseline=13mm]
\coordinate (a) at(1,0);
\coordinate (b) at(0,1);
\coordinate (c) at(2,1);
\coordinate (d) at(0,2.5);
\coordinate (e) at(2,2.5);
\draw (a)--(b);
\draw (a)--(c);
\draw (b)--(d);
\draw (b)--(e);
\draw (c)--(d);
\draw (c)--(e);
\node [below =0.1cm of a] {$1$};
\node [left =0.2cm of b] {$x$};
\node [right =0.2cm of c] {$y$};
\node [above left=0.1cm and -0.3cm of d] {$xx=xy$};
\node [above right =0.1cm and -0.3cm of e] {$yx=yy$};
\foreach \x in {a,b,c,d,e}
\draw[fill=black] (\x) circle(0.1);
\end{tikzpicture}
\end{minipage}

\begin{rque}
Using the fact that the algebra/operad is basic-set, we will sometimes label edges rather than vertices by elements of the algebra/operad. The previous example will be drawn as:
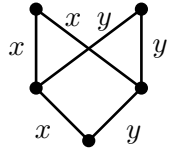
\begin{center}
\begin{tikzpicture}[line width=1pt, scale=0.7, baseline=13mm]
\coordinate (a) at(1,0);
\coordinate (b) at(0,1);
\coordinate (c) at(2,1);
\coordinate (d) at(0,2.5);
\coordinate (e) at(2,2.5);
\draw (a)edge node[below left]{$x$}(b);
\draw (a)edge node[below right]{$y$}(c);
\draw (b)edge node[left]{$x$}(d);
\draw (b)edge node[above=-0.07cm, pos=0.65]{$y$}(e);
\draw (c)edge node[above, pos=0.65]{$x$}(d);
\draw (c)edge node[right]{$y$}(e);
\foreach \x in {a,b,c,d,e}
\draw[fill=black] (\x) circle(0.1);
\end{tikzpicture}
\end{center}
So that the reading of a chain from bottom to top gives an element in the algebra read from right to left.
\end{rque}

\subsubsection{Operads} \label{expleOperad}

Let us consider the commutative operad $\com$. A presentation of this operad as a symmetric operad is
\[
\com= \Tc(\mu) / \left((\mu \circ (\operatorname{id} , \mu) - \mu \circ (\mu , \operatorname{id})  ). \Sb_3\right),
\]
with $\mu$ in arity $2$ invariant under the action of the symmetric group $\Sb_2$.
It is a basic-set operad, with a unique element denoted by $e_n$ in every positive arity $n$. 
The composition of operads is given by $\gamma(e_p , e_{k_1} ,  \ldots , e_{k_p}) = e_{k_1+\cdots+k_p}$. 
A presentation of $\com_{sh}$ is then
\[
\com_{sh}= \Tc_{sh}(\mu) /  (\mu \circ (\operatorname{id} \otimes \mu) - \mu \circ (\mu \otimes \operatorname{id}),\ \mu \circ (\operatorname{id} \otimes \mu) -  \mu \circ_{1,(23)} \mu ),
\]
where we denote again by $\mu$ the image of $\mu$ by the functor $(-)_{sh}$. 
The associated partition posets are the usual partition posets, 
ordered by refinement, i.e. two partitions on $n$ elements $A$ and 
$B$ satisfy $A \leq B$ if and only if parts of $A$ are included in 
parts of $B$. 
Figure \ref{partitionPosetCom} represents the partition poset in 
arity $3$ associated with the operad $\com$.

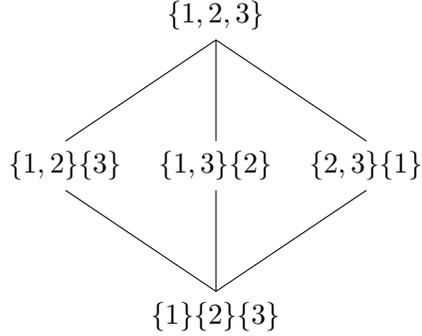
\begin{figure}[h]
\begin{tikzpicture}[scale=0.5]
\def\e{4cm}
\def\z{4cm}
\node (min2) at (0,0) {$\{1,2,3\}$};
\node (a) at (0,-4) {$\{1,3\}\{2\}$};
\node (b) at (-\z,-4) {$\{1,2\}\{3\}$};
\node (c) at (\z,-4) {$\{2,3\}\{1\}$};

\node (max) at (0,-8) {$\{1\}\{2\}\{3\}$};

\draw (min2.south)edge (a.north);
\draw (min2.south)edge (b.north);
\draw (min2.south)edge (c.north);

\draw (a.south)edge (max.north);
\draw (b.south)edge (max.north);
\draw (c.south)edge (max.north);
\end{tikzpicture}
\caption{Partition poset in arity $3$ associated with the operad $\com$}\label{partitionPosetCom}
\end{figure}

\bigskip

Let us now consider the permutative operad $\perm$, introduced by Chapoton in \cite{ChaPerm}. A presentation of this operad as a symmetric operad is 
\[
\perm= \Tc(\mu.\Sb_2) /
\left( \left(\mu \circ_1 \mu - \mu \circ_2 \mu,\ \mu \circ_1 \mu - (\mu \circ_2 \mu) . (23)\right).\Sb_3\right),
\]
 where $\mu .\Sb_2$ is in arity $2$ and endowed with the regular action of $\Sb_2$. 
It is a basic-set operad whose operations in arity $n$ are the pointed versions of the set $\{1, \ldots, n\}$.
Denoting by $e_{n,k}$, with $1 \leq k \leq n$ the elements of $\perm(n)$, the operadic composition is given by $\gamma(e_{p, a} \otimes e_{k_1, b_1} \otimes  \cdots \otimes e_{k_p, b_p}) = e_{k_1+\cdots+k_p, k_1+ \cdots +k_{a-1}+b_a}$. 
A presentation of $\perm_{sh}$ is then 
\[
\perm_{sh}= \Tc_{sh}(\mu, \mu^{\tau}) / (R_{sh}),
\]
where
\begin{align*}
R_{sh} \coloneqq & (\mu \circ_1 \mu - \mu \circ_{1,(23)} \mu,\ \mu \circ_1 \mu - \mu \circ_2 \mu,\ \mu \circ_1 \mu - \mu \circ_2 \mu^{\tau},\\
& \mu \circ_{1,(23)} \mu^{\tau} - \mu^{\tau} \circ_2 \mu^{\tau},\ \mu \circ_{1,(23)} \mu^{\tau} - \mu^{\tau} \circ_1 \mu,\ \mu \circ_{1,(23)} \mu^{\tau} - \mu^{\tau} \circ_1 \mu^{\tau},\\
& \mu^{\tau} \circ_2 \mu - \mu \circ_1 \mu^{\tau},\ \mu^{\tau} \circ_2 \mu - \mu^{\tau} \circ_{1,(23)} \mu^{\tau},\ \mu^{\tau} \circ_2 \mu - \mu^{\tau} \circ_{1,(23)} \mu),
\end{align*}
and $\mu^{\tau}$ is the image of $\mu.(12)$ by the functor $(-)_{sh}$.

The associated partition posets are the pointed partition posets studied by Chapoton and Vallette in \cite{ChVal}. The underlying partitions are partitions of $\{1, \ldots, n\}$, with a distinguished element in every part and the order is given by refinement, as in the usual partition case, with a compatibility of distinguished elements (if an element $x$ is distinguished in a permutation $B$, it must also be in $A$ finer than $B$ to have $A \leq B$). Figure \ref{partitionPosetPerm} represents the partition posets in arity $3$ associated with the operad $\perm$.

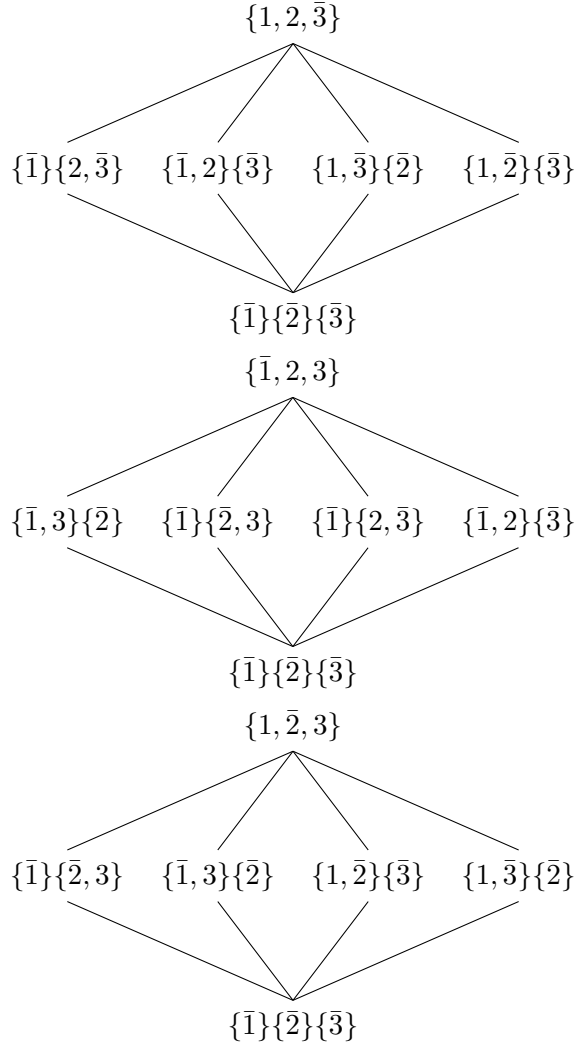
\begin{figure}[h!]
\begin{tikzpicture}[scale=0.5]
\def\e{4cm}
\def\z{4cm}
\node (min3) at (4,0) {$\{1,2,\bar{3}\}$};
\node (m2) at (-\z/2,-4) {$\{\bar{1}\}\{2,\bar{3}\}$};
\node (d) at ([xshift=\z]m2) {$\{\bar{1},2\}\{\bar{3}\}$};
\node (c) at ([xshift=\z]d) {$\{1,\bar{3}\}\{\bar{2}\}$};
\node (e) at ([xshift=\z]c) {$\{1,\bar{2}\}\{\bar{3}\}$};
\node (max) at (4,-8) {$\{\bar{1}\}\{\bar{2}\}\{\bar{3}\}$};

\draw (min3.south)edge (e.north);
\draw (min3.south)edge (m2.north);
\draw (min3.south)edge (d.north);
\draw (min3.south)edge (c.north);

\draw (c.south)edge (max.north);
\draw (d.south)edge (max.north);
\draw (e.south)edge (max.north);
\draw (m2.south)edge (max.north);
\end{tikzpicture}
\begin{tikzpicture}[scale=0.5]
\def\e{4cm}
\def\z{4cm}
\node (min1) at (-4,0) {$\{\bar{1},2,3\}$};
\node (m2) at (-\z/2,-4) {$\{\bar{1}\}\{2,\bar{3}\}$};
\node (a) at ([xshift=-\z]m2) {$\{\bar{1}\}\{\bar{2},3\}$};
\node (b) at ([xshift=-\z]a) {$\{\bar{1},3\}\{\bar{2}\}$};
\node (d) at ([xshift=\z]m2) {$\{\bar{1},2\}\{\bar{3}\}$};
\node (max) at (-4,-8) {$\{\bar{1}\}\{\bar{2}\}\{\bar{3}\}$};

\draw (min1.south)edge (m2.north);
\draw (min1.south)edge (a.north);
\draw (min1.south)edge (b.north);
\draw (min1.south)edge (d.north);

\draw (a.south)edge (max.north);
\draw (b.south)edge (max.north);
\draw (d.south)edge (max.north);
\draw (m2.south)edge (max.north);
\end{tikzpicture}
\begin{tikzpicture}[scale=0.5]
\def\e{4cm}
\def\z{4cm}
\node (min2) at (0,0) {$\{1,\bar{2},3\}$};
\node (m2) at (-\z/2,-4) {$\{\bar{1},3\}\{\bar{2}\}$};
\node (a) at ([xshift=-\z]m2) {$\{\bar{1}\}\{\bar{2},3\}$};
\node (d) at ([xshift=\z]m2) {$\{1,\bar{2}\}\{\bar{3}\}$};
\node (c) at ([xshift=\z]d) {$\{1,\bar{3}\}\{\bar{2}\}$};
\node (max) at (0,-8) {$\{\bar{1}\}\{\bar{2}\}\{\bar{3}\}$};

\draw (min2.south)edge (d.north);
\draw (min2.south)edge (a.north);
\draw (min2.south)edge (m2.north);
\draw (min2.south)edge (c.north);

\draw (a.south)edge (max.north);
\draw (c.south)edge (max.north);
\draw (d.south)edge (max.north);
\draw (m2.south)edge (max.north);
\end{tikzpicture}
\caption{Partition posets in arity $3$ associated with the operad $\perm$}\label{partitionPosetPerm}
\end{figure}

\subsection{CL-labelling and isomorphisms of subposets}

For $I$ and $J$ finite sets, any bijective map $f : I \to J$ induces a bijection 
\[ \tPc(f) : \tPc (I) \to \tPc (J),\ \overline{\nu_t \times (x_1,\, \ldots ,\, x_t)} \mapsto \overline{\nu_t \times (f(x_1),\, \ldots ,\, f(x_t))}. \]

\begin{defi}
Let $\Pi_1$ and $\Pi_2$ be two (pure) interval subposets of the partition posets $\Pi_{\tPc}^{(d_1)}$ and $\Pi_{\tPc}^{(d_2)}$ respectively. 
For $i \in \{ 1,\, 2\}$, let
\[
E_i \coloneqq \bigcup_{\substack{\lambda,\, \omega \in \Pi_i\\ \lambda \leq \omega}} D_{\lambda}^{\omega}.
\]
\begin{enumerate}
\item
We say that the two subposets $\Pi_1$ and $\Pi_2$ are \emph{isomorphic} if there is an isomorphism $g : \Pi_1 \to \Pi_2$ of posets and an increasing (for the usual total order on $\Nb$) bijection $f : E_1 \to E_2$ such that for every edge edge $\lambda \prec \omega$, we have
\[
g(\delta_{\lambda}^{\omega}) = \tPc (f)(\delta_{\lambda}^{\omega}).
\]
\item
Assume that every poset $\Pi_{\tPc}^{(d)}$ is endowed with a CL-labelling. We say that the CL-labellings on $\left\{ \Pi_{\tPc}^{(d)}\right\}_d$ are \emph{compatible with isomorphisms of subposets} if when $\Pi_1$ is a subposet of $\Pi_{\tPc}^{(d_1)}$ isomorphic to a subposet $\Pi_2$ of $\Pi_{\tPc}^{(d_2)}$, the isomorphism induces a map on the labels of the CL-labellings sending increasing chains to increasing chains, non-increasing chains to non-increasing chains, and preserving the preorder on chains.
\end{enumerate}
\end{defi}

\begin{rque}
The intervals drawn on Figure \ref{partitionPosetPerm} are not isomorphic for our definition, but for instance, the first one is isomorphic to the subposet of $\Pi_6$ drawn on Figure \ref{toto}.
\begin{figure}
\begin{center}
\begin{tikzpicture}[scale=0.5]
\def\e{4cm}
\def\z{6cm}
\node (min3) at (4,0) {$\{1,2,3,\bar{4},5,6\}$};
\node (m2) at (-\z/2,-4) {$\{\bar{1}\}\{2,3,\bar{4},5,6\}$};
\node (d) at ([xshift=\z]m2) {$\{\bar{1},2,5,6\}\{3,\bar{4}\}$};
\node (c) at ([xshift=\z]d) {$\{1,3,\bar{4}\}\{\bar{2},5,6\}$};
\node (e) at ([xshift=\z]c) {$\{1,\bar{2},5,6\}\{3,\bar{4}\}$};
\node (max) at (4,-8) {$\{\bar{1}\}\{\bar{2},5,6\}\{3,\bar{4}\}$};

\draw (min3.south)edge (e.north);
\draw (min3.south)edge (m2.north);
\draw (min3.south)edge (d.north);
\draw (min3.south)edge (c.north);

\draw (c.south)edge (max.north);
\draw (d.south)edge (max.north);
\draw (e.south)edge (max.north);
\draw (m2.south)edge (max.north);
\end{tikzpicture}
\end{center}
\caption{This poset is isomorphic to the first poset of Figure \ref{partitionPosetPerm}}\label{toto}
\end{figure}
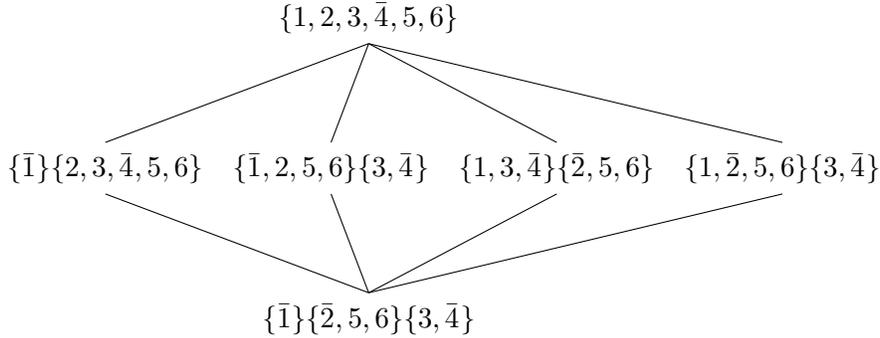 
\end{rque}

\section{Main theorem}\label{Section: main theorem}

In this section, we will consider operadic partition posets endowed with compatible CL-labellings which make these posets CL-shellable.

\subsection{From compatible CL-shellability to PBW bases}

Let us consider a basic-set operad $\tPc$, of the form $\Tc(E)/(R)$ where $R$ is quadratic, and its associated partition posets $\Pi_{\tPc}$. We suppose that these posets are CL-shellable with CL-labellings being compatible with the subposet isomorphisms as defined in the previous subsection.\\

We need the following definition to define a confluent rewriting system on elements of $\Tc(E)$:
\begin{defi}
We say that a maximal chain in a poset is \emph{adjacent} to an other one if it only differs from the first one by two edges. The poset obtained by the set of differing edges between the two chains is then called \emph{diamond}.
\end{defi}

The picture below depicts two adjacent chains. The diamond is the coloured part.

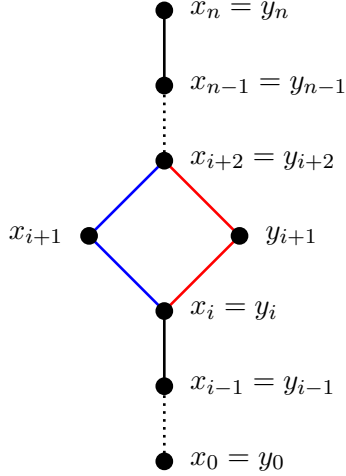
\begin{center}
\begin{tikzpicture}[line width=1pt]
\coordinate (a) at(0,0);
\coordinate (b) at(0,1);
\coordinate (c) at(0,2);
\coordinate (d) at(-1,3);
\coordinate (e) at(1,3);
\coordinate (f) at(0,4);
\coordinate (g) at(0,5);
\coordinate (h) at(0,6);
\draw[dotted] (a)--(b);
\draw[dotted] (f)--(g);
\draw (b)--(c);
\draw (g)--(h);
\draw[blue] (c)--(d)--(f);
\draw[red] (c)--(e)--(f);
\node [right =0.2cm of a] {$x_0=y_0$};
\node [right =0.2cm of b] {$x_{i-1}=y_{i-1}$};
\node [right =0.2cm of c] {$x_i=y_i$};
\node [left =0.2cm of d] {$x_{i+1}$};
\node [right =0.2cm of e] {$y_{i+1}$};
\node [right =0.2cm of f] {$x_{i+2}=y_{i+2}$};
\node [right =0.2cm of g] {$x_{n-1}=y_{n-1}$};
\node [right =0.2cm of h] {$x_n=y_n$};
\foreach \x in {a,b,c,d,e,f,g,h}
\draw[fill=black] (\x) circle(0.1);
\end{tikzpicture}
\end{center}

Let us first describe a partial order on elements of $\Tc(E)$. This order will be defined from the partial preorder on chains given by the lexicographic order on CL-labellings. Denote by $\pi : \mathcal{N}(E) \to \Tc(E)$ the map obtained by forgetting the levels (see Section \ref{normalisedBarConstruction} for a definition of $\Nc(E)$), and by $\cdot$ the concatenation of chains in the poset $\Pi_{\widetilde \Pc}$. We recall from Remark \ref{maximalChains} that maximal chains of $\Pi_{\tPc}^{(d)}$ coincide with $\Nc(E)^{(d)}$. 

\begin{defi}
We say that there is an \emph{exchange relation} between the two elements $\tilde a = \nu^1 \circ_{i_1} \nu^2 \circ_{i_2} \cdots \circ_{i_{l-1}} \nu^l$ and $\tilde b = \mu^1 \circ_{j_1} \mu^2 \circ_{j_2} \cdots \circ_{j_{m-1}} \mu^m$ in $\Nc(E)$ if
\[
\pi(\tilde a) = \pi (\tilde b)
\]
and if there exists $k \in \llbracket 1,\, l-1\rrbracket$ such that
\[
\nu^s = \mu^s \text{ and }  i_s = j_s \text{ for all } s \in \llbracket 1,\, l\rrbracket \backslash \{ k,\, k+1\}.
\]
In particular, this implies that
\[
l = m,\ \nu^k = \mu^{k+1} \text{ and } \nu^{k+1} = \mu^k.
\]
\end{defi}

\begin{exple}
Consider a binary generator $\mu$. There is an exchange relation between $\tilde a = \mu \circ_2 \mu \circ_1 \mu$ and $\tilde b = \mu \circ_1 \mu \circ_3 \mu$. Pictorially, we have
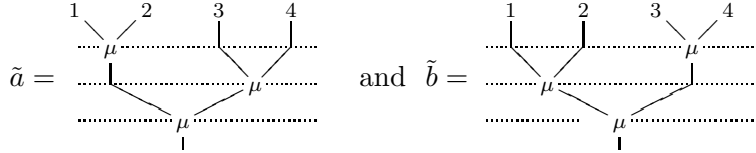
\[
\tilde a = \vcenter{\scriptsize{
\xymatrix@R=5pt@C=5pt@M=2pt{1\ar@{-}[dr] &&
2 \ar@{-}[dl] && 3 \ar@{-}[d] && 4 \ar@{-}[d] &&\\
\ar@{.}[r] & \mu \ar@{.}[rrr] \ar@{-}[d] &&&  *{} \ar@{-}[dr] \ar@{..}[rr] &&  *{} \ar@{-}[dl]\ar@{.}[r] &\\
\ar@{.}[r] & *{} \ar@{.}[rrrr] \ar@{-}[drr] &&&& \mu \ar@{.}[rr] \ar@{-}[dll] &&\\
\ar@{.}[rrr] &&& \mu \ar@{.}[rrrr] \ar@{-}[d] &&&&\\
&&&&&&&\\ }}}
\text{ and }\
\tilde b = \vcenter{\scriptsize{
\xymatrix@R=5pt@C=5pt@M=2pt{& 1\ar@{-}[d] &&
2 \ar@{-}[d] && 3 \ar@{-}[dr] && 4 \ar@{-}[dl] &&\\
\ar@{.}[r] & *{} \ar@{-}[dr] \ar@{.}[rrr] && *{} \ar@{-}[dl] \ar@{.}[rrr] &&&  \mu \ar@{-}[d] \ar@{..}[rr] &&\\
\ar@{.}[rr] && \mu \ar@{.}[rrrr] \ar@{-}[drr] &&&& *{} \ar@{.}[rr] \ar@{-}[dll] &&\\
\ar@{.}[rrr] &&&& \mu \ar@{.}[rrrr] \ar@{-}[d] &&&&\\
&&&&&&&&\\ }}}
\]
\end{exple}

Consider two trees $a$ and $b$ in $\Tc(E)$. Let us first remark that if there are two adjacent (maximal) chains $\tilde{a}$ and $\tilde{b}$ such that $a=\pi(\tilde{a})$ and $b=\pi(\tilde{b})$, then the diamond between $\tilde{a}$ and $\tilde{b}$ comes from one quadratic relation in the operad or from an exchange relation (this second case appears if and only if $a=b$). We will deal with the first type of diamond to define an order on  $\Tc(E)$.

\begin{defi}
For any two elements $a \neq b$ in $\Tc(E)$, we define $a \lessdot b$ if there exist two adjacent chains $\tilde{a}= r \cdot (g \prec x \prec h) \cdot s$ and $\tilde{b}= r \cdot (g \prec y \prec h) \cdot s$ such that:
\begin{itemize}
\item  $a=\pi(\tilde{a})$ and $b=\pi(\tilde{b})$,
\item the CL-labelling given by $\lambda_r^{g,\, h}$ (see Remark \ref{remark: restricted CL-labelling}) of the chain $(g \prec x \prec h)$ is the unique increasing chain (minimal in the lexicographic order) in this interval.
\end{itemize}
\end{defi}

\begin{lem}
This binary relation is anti-symmetric as soon as the partition posets $\Pi_{\tPc}$ admit CL-labellings compatible with the isomorphisms of subposets.
\end{lem}

\begin{proof}
Let $a$ and $b$ be two distinct elements in $\Tc(E)$. If there exist two pairs of adjacent chains $(\tilde a,\, \tilde b)$ and $(\bar a,\, \bar b)$, with every first component projecting to $a$ and second component projecting to $b$, then the associated diamonds correspond to the same relation in $R$. (Indeed, elements in $\pi^{-1}(a)$ only differ from each other by a sequence of exchange relations.) The diamonds form therefore two isomorphic subposets of $\Pi_{\widetilde \Pc}$. The compatibility between CL-labellings implies that the lexicographic order on CL-labellings of two adjacent chains $\tilde a$ and $\tilde b$ does not depend on the pair of lifting of $a$ and $b$: it allows us to define the previous binary relation on such $a$ and $b$ in $\Tc(E)$. It is anti-symmetric since the increasing chain is always minimal in the lexicographic order.
\end{proof}

\begin{lem}\label{lemOrdre}
Assume that the partition posets $\Pi_{\tPc}$ admit CL-labellings compatible with the isomorphisms of subposets. The reflexive and transitive closure of the relation $\lessdot$, denoted by $\leq$, satisfies that 
\begin{equation}\label{compRel}
a < b \implies \min \left(\left\{\tilde{a}| \pi(\tilde{a})=a\right\} \right) <_{lex} \min \left( \left\{\tilde{b}| \pi(\tilde{b})=b \right\} \right),
\end{equation}
where 
\begin{itemize}
 \item the relation $<_{lex}$ is the order on chains given by the lexicographic order on the associated CL-labelling,
 \item the minimum is taken according to this order,
\item only maximal chains are considered.
\end{itemize}
 Hence, $\leq$ is a well-defined partial order.
\end{lem}

\begin{proof}
By the transitivity of $<_{lex}$, it is enough to prove the implication (\ref{compRel}) for $a \lessdot b$. We use a reductio ad absurdum. Let us consider that there exist two elements $a \neq b$ in $\Tc(E)$ satisfying $a \lessdot b$ and that there exists a maximal chain $\tilde{b}$ projecting to $b$ smaller than any chain $\tilde{a}$ projecting to $a$. As $a \lessdot b$, the two trees only differs by a (set) quadratic relation, say $\mu^1 \circ_i \mu^2 = \nu^1 \circ_j \nu^2$ with $\mu^1 \circ_i \mu^2$ a subtree of $a$ and $\nu^1 \circ_j \nu^2$ a subtree of $b$. We construct a maximal chain $\hat{b}$ such that $\pi(\hat{b}) = b$ as follows: $\hat{b}$ coincides with $\tilde{b}$ until the level which contains $\nu^2$ (be aware that we read $\tilde{b}$ from the top level to the bottom level), the next level is given by means of $\nu^1$ and the other levels can be chosen without other condition than $\pi(\hat{b}) = b$. Similarly, let $\hat{a}$ be the maximal chain projecting to $a$ defined by: $\hat{a}$ coincides with $\tilde{b}$ until the level before the level which contains $\nu^2$, the next level is given by means of $\mu^2$, the one after is given by means of $\mu^1$ and the other levels are chosen as for $\hat{b}$. We get the following picture.
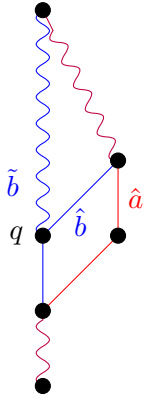
\begin{center}
\begin{tikzpicture}
\coordinate (p1) at (0,0);
\coordinate (p2) at (0,1);
\coordinate (p3) at (0,2);
\coordinate (p4) at (1,2);
\coordinate (p5) at (1,3);
\coordinate (p6) at (0,5);
\draw[purple, snake=coil,segment aspect=0] (p1)--(p2);
\draw[blue](p3) edge[snake=coil,segment aspect=0] node[shift={(-0.4, -0.8)}]{$\tilde{b}$} (p6);
\draw[blue] (p2)--(p3) edge node[below]{$\hat{b}$} (p5);
\draw[purple, snake=coil,segment aspect=0](p5)--(p6);
\draw[red] (p2)--(p4) edge node[right]{$\hat{a}$} (p5);
\draw[fill=black] (p1) circle(0.1);
\draw[fill=black] (p2) circle(0.1);
\draw (-0.35,2) node{$q$};
\draw[fill=black] (p3) circle(0.1);
\draw[fill=black] (p4) circle(0.1);
\draw[fill=black] (p5) circle(0.1);
\draw[fill=black] (p6) circle(0.1);
\end{tikzpicture}
\end{center}

Then, the exchange relations to go from $\tilde b$ to $\hat b$ change the CL-labelling only above the element $q$ and the change of the CL-labelling corresponding to the relation $a < b$ modifies the label corresponding to the edge below $q$. This would give a CL-labelling for $\hat{a}$ smaller than the CL-labelling of $\tilde b$, which contradicts the second part of the hypotheses. Thus the first point of the lemma holds. 

As the lexicographic order is anti-symmetric, the binary relation $\leq$ on $\Tc(E)$ also is. It is also reflexive and transitive by definition, hence the result.
\end{proof}

Let us now give a characterisation for minimal elements in $\Tc(E)$ in terms of their representatives in $\Nc(E)$.
Recall that elements of $\Tc(E)$ have a unique shuffle representation using rooted planar trees. We consider in the sequel only the unique shuffle (planar) representation of the trees.

Assume that the partition posets $\Pi_{\tPc}$ admit CL-labellings compatible with the isomorphisms of subposets and let $\mu^1 \circ_i \mu^2$ be a term in a quadratic relation of $\tPc$. By means of the isomorphisms of subposets, when this relation appears as two consecutive levels in a maximal chain, the fact that the corresponding CL-labelling is increasing or not only depends on the quadratic term.
We first show the following lemma.

\begin{lem}\label{lem: increasing chain}
Let $a$ be a tree in $\Tc(E)$. Assume that the partition posets $\Pi_{\tPc}$ admit CL-labellings compatible with the isomorphisms of subposets and that the CL-labellings corresponding to every quadratic subtree in $a$ are increasing. Then there exists a maximal chain in $\pi^{-1}(a) \subset \Nc(E)$ whose CL-labelling is the minimal increasing chain.
\end{lem}

\begin{proof}
We prove the result by induction on the height $d$ of the maximal chains in $\pi^{-1}(a)$.

{\bf Base cases}: there is nothing to prove for $d=1$ and the case $d=2$ is obvious.

{\bf Inductive step}: assume that $d \geq 3$. We suppose that the result holds for all maximal chains of height $< d$. Let $\tilde{a}$ be in $\pi^{-1}(a)$. We write $\tilde{a} = \hat{a} \circ_i \mu$. By the induction hypothesis, there exists a maximal chain $\hat{a}'$ whose CL-labellings is the minimal increasing chain and such that $\hat{a}' \circ_i \mu \in \pi^{-1}(a)$. If the CL-labelling associated with $\hat{a}' \circ_i \mu$ is not the increasing chain, we write $\hat{a}' = \hat{a}'' \circ_j  \mu'$. We know that the two last levels of $\hat{a}' \circ_i \mu$, given by $\mu'$ and $\mu$, do not correspond to a quadratic relation otherwise the associated CL-labelling would be increasing by assumption on $a$ and the CL-labelling associated with $\hat{a}' \circ_i \mu$ would be too.

We can therefore consider the exchange relation associated with these two levels to get the levelled tree $\hat{a}'' \circ_{i'} \mu \circ_{j'} \mu'$. The CL-labelling associated with the two last levels is now increasing. If the CL-labelling associated with $\hat{a}'' \circ_{i'} \mu$ is non-increasing, we continue similarly. Otherwise, we stop. We finally get $\bar{a} \circ_k \mu \circ_l \tilde{S}$ where $S$ can be seen as a forest of intertwined levelled trees coming from a forest $S$ of trees in $\Tc(E)$. 

By the induction hypothesis, for any tree $b$ in the forest $S$, there exists a levelled tree $\tilde{b}$ projecting to $b$ by $\pi$ and such that the associated CL-labelling is increasing. It is possible to intertwine the levelled trees $\tilde{b}$ in order to get a forest of intertwined levelled trees $\hat{S}$ associated with the same forest $S$ as $\tilde{S}$ and such that the corresponding CL-labelling (rooted by $\bar{a} \circ_k \mu$) is increasing. Indeed, let us denote by $\hat{S}'$ the forest of intertwined levelled trees defined as follows: the first levels are given by the levelled tree $\tilde{b}_1$ corresponding to a first tree $b_1$ in the forest $S$, then we put the levelled tree $\tilde{b}_2$ corresponding to a second tree $b_2$ in $S$ and so on. If the CL-labelling associated with the first term of $\tilde{b}_2$ and the last term of $\tilde{b}_1$ is non-increasing, we make use of an exchange relation in order to get an increasing CL-labelling. We continue until we get an increasing CL-labelling. We do the same reasoning with the second term in $\tilde{b}_2$ and so on and so forth and for the other levelled trees $\tilde{b}_k$. We finally get the wanted forest of intertwined levelled trees $\hat{S}$.

We conclude by saying that the CL-labelling associated with $\bar{a} \circ_k \mu \circ_l \hat{S}$ is increasing since the CL-labelling associated with the levelled tree corresponding to $\mu$ and the first level of $\hat{S}$ is increasing by construction of $\tilde{S}$.
\end{proof}

We are now able to prove the following characterisation.

\begin{lem} \label{formNormale}
An element in $\Tc(E)$ is minimal for the previously defined order $\leq$ if and only if one of its representatives in $\Nc(E)$ is the minimal increasing chain in the interval.
\end{lem}

\begin{proof}
\begin{description}
\item[$\Leftarrow$] We first assume that some $\tilde a \in \Nc(E)$ representing $a \in \Tc(E)$ is the minimal increasing chain. Using Lemma \ref{lemOrdre}, the element $a$ is minimal otherwise there would exist a chain smaller than $\tilde{a}$, which contradicts the hypotheses.

\item[$\Rightarrow$] To prove the converse direction, we use a proof by contraposition. Let $a$ be an element in $\Tc(E)$ such that none of the representatives $\tilde{a}$ in $\Nc(E)$ is the minimal increasing chain (for the CL-labelling). By Lemma \ref{lem: increasing chain}, we get that there exists (at least) one quadratic subtree $q$ in $a$ and a quadratic tree $r$ such that $r \lessdot q$. It follows that the tree $b$ obtained by means of $a$ where the subtree $q$ is replaced by $r$ satisfies $b \lessdot a$. This shows that $a$ is not minimal.
\end{description}
\end{proof}

\begin{thm} \label{ThmCLPBW}
Let $\widetilde \Pc$ be a quadratic basic-set operad and let $\Pi_{\widetilde \Pc}$ be the associated operadic partition posets. We assume that the posets $\left\{\Pi_{\widetilde \Pc}^{(d)}\right\}_d$ admit CL-labellings compatible with isomorphisms of subposets. 

Then, the algebraic operad $\Pc = \Tc(E)/(R)$ associated to $\widetilde \Pc$ admits a PBW basis with a partial order, as defined in Definition \ref{defPBW}.
\end{thm}

\begin{proof}
We previously defined a partial order (see Lemma \ref{lemOrdre}). This partial order gives a rewriting system on $\Tc(E)$ by defining rewriting rules $a \rightarrow b$ for any covering relations $a>b$.
Moreover the CL-compatibility ensures that the rewriting rules are context-free, that is, if $a \rightarrow b$, then for every $u$ and $v$ in $\Tc(E)$ and every pointed shuffles $w,w'$ we have $u \circ_{i,w} a \circ_{j,w'} v  \rightarrow u \circ_{i,w} b \circ_{j,w'} v$. 
The obtained order is thus compatible with the shuffle composition.

As the lexicographic order is decreasing at every rewriting step and as the number of generators (and thus of chains of a fixed weight) is finite, the rewriting system is terminating.

Before studying confluence, let us observe the crucial fact that normal forms in $\Tc(E)$ are exactly minimal elements for the previously defined partial order $\leq$. Hence, according to Lemma \ref{formNormale}, an element in $\Tc(E)$ is a normal form if and only if one of its representatives in $\Nc(E)$ is the minimal increasing chain in the interval.

We now show the confluence of critical pairs. 

Let $a$ be in $\Tc(E)$ be the source of a critical pair. It implies that $a$ is in weight $3$ (as relations are quadratic in the operad) and that 
$a$ can be rewritten as two different elements $a'_0$ and $a''_0$, which can be each further rewritten into normal forms $a'$ and $a''$ respectively 
(because of termination of the rewriting system in $\Tc(E)$). 

Let us consider the partition poset $P_a$ associated with the element $a$ seen as an element of the operad. 
Among all chains $\tilde{a'}$ representing $a'$, one of them is the minimal increasing one in  $P_a$, as $a'$ cannot be further rewritten in $\Tc(E)$.
The same reasoning applies to $a''$: some chain $\tilde{a''}$ representing $a''$ is the minimal increasing one in $P_a$.
Therefore $a'=a''$.

This proves the confluence of the critical pair of source $a$.

Then we have a convergent rewriting system on $\Tc(E)$. The first point of Definition \ref{defPBW} comes from the fact that the order is decreasing with rewriting steps and the second point, from the fact that the operad is quadratic, hence all rewriting steps come from rewriting steps on product of generators of the operad.
\end{proof}

\begin{rque}
The result is different from the result of Quinn in \cite{dQ13} which studies posets and their incidence algebras, even if it looks similar at first sight. Namely, starting with a poset, he shows that the poset is LEX-shellable if and only if the incidence algebra is PBW for some specific lexicographic order. There is no obvious link between our results, as the associations between a poset and an algebraic object (algebra or operad) differ in the two papers.
\end{rque}

\subsection{Example: Commutative and permutative case}

We described in Section \ref{expleOperad} the partition posets associated with the operads $\com$ and $\perm$. Let us illustrate the previous main theorem with these posets.

First, let us start with partition posets associated with $\com$, i.e. usual partition posets. A CL-labelling (EL-labelling in fact, i.e. CL-shellable, with the label of a given edge being independent of the chain containing it) of these posets was given in \cite{ShBjorn} by Bj{\"o}rner, following an idea of Gessel. The edge between the partitions $(A_1, \ldots, A_p)$ and $(A_1, \ldots, \hat{A_i}, \ldots, \hat{A_j}, \ldots,\, A_p, \allowbreak A_i \cup A_j)$ is given by
\[
\delta = \vcenter{
\xymatrix@M=0pt@R=4pt@C=4pt{\min(A_i)&& \min(A_j)\\
& *{} \ar@{-}[lu] \ar@{-}[ur] \ar@{-}[d] &\\
&&}}
\]
and is labelled by $\max(\min(A_i),\min (A_j))$. 

\begin{prop} \label{compCL} This labelling is compatible with isomorphisms of subposets.
\end{prop}

\begin{proof}
The condition on considered isomorphisms of subposets is given on labellings by an increasing bijection $f$. We have therefore $\min(A_i)<\min(A_j)$ implies $\min(f(A_i))<\min(f(A_j))$. It is equivalent to the property:
\begin{equation}
\max(\min(f(A_i)),\min(f(A_j))) = f(\max(\min(A_i), \min( A_j))).
\end{equation}
The labellings depend only on the minimal elements of the poset and any subposet isomorphism preserves the order on minimal elements by definition: 
 the CL-labellings are thus compatible with isomorphisms of subposets. 
\end{proof}

We can then apply Theorem \ref{ThmCLPBW} to get a PBW basis of the commutative operad. Indeed, elements of the PBW basis are obtained as increasing chains for the CL-labelling. They are thus obtained by merging the two parts having the smallest minimal elements.
\begin{cor}
The PBW basis associated to this CL-labelling is the basis of binary left-side combs with entries labelled from left to right by $1$, \ldots, $n$.
\end{cor}

\section{Study of the converse of the main theorem}

In this section, we present two examples related to the converse 
direction of Theorem \ref{ThmCLPBW}. The second (counter)-example
is the counter-example of the theorem, that is a PBW algebra admitting no PBW basis which could be associated with an isomorphism compatible CL-shelling. To make it more accessible, we first introduce an example on which it is based.

\subsection{First counter-example}

Let us begin with a first example on algebras which proves the complexity of the converse of Theorem \ref{ThmCLPBW}. We consider the algebra in sets on the generators $a$, $b$, $c$, $d$, $e$, $f$,\, \ldots ,\, $l$ with relations $da=eb$, $fb=hc$, $gb=ic$, $kh=li$ and $je=kf=lg$. 
\begin{prop} \label{basic1}
This algebra is basic-set.
\end{prop}

\begin{proof}
Suppose that it is not the case and there exists two words $u\neq v$ 
and a letter $x$ such that $ux=vx$. Without loss of generalities,
we can assume that $u$ and $v$ are of minimal length, let us say of 
length $n$ for instance.

\begin{itemize}
\item Let us first assume that $u_n \neq v_n$. We presume that 
$u=u_1\ldots u_n$ and $v=v_1\ldots v_n$. Then $x$ must be the last 
letter of at least one relation, otherwise $ux=vx$ would imply 
$u=v$ which is assume to be false. Moreover, $u_n \neq v_n$ implies 
that $x$ is in fact in at least two relations: it is then $b$ or 
$c$.

The relation interfering here can then be $da=eb$, $fb=hc$ or 
$gb=ic$.
However there are no other relations using $d$ so that its use 
could not lead to a relation $ux=vx$.

Let us try to build $u$ and $v$: if $x=b$, both relations $fb=hc$ 
and $gb=ic$ must be used to avoid $u_n=v_n$. For instance, let us 
say that $u_n=f$ and $v_n=g$. When rewriting $ux$ in $vx$, we 
should then use a relation of the type $*h=\ldots$. Iterating this 
reasoning, the only possible word for $ux$ is $kfb$ and $vx$ is $lgb$.
 However $kf=lg$: we get a contradiction.
 
 If $x=c$, the only possibility is $ux=khc$ and $vx=lic$, but once again $kh=li$.

\item Let us now assume that $u_n=v_n$. We can take $u$ and $v$ such that the first and the last relation involve $x$. Then $u_nx$ must be equal to a $\alpha \beta$, $\alpha$ being the right element of a product belonging to some relations. The possibilities are:
\begin{itemize}
\item $u_nx=fb$ or $hc$: $ux$ is then  $kfb$ or $khc$. It cannot be longer as there are no relations of type $\alpha k$. It is moreover not possible to rewrite it in some $\gamma fb$ or $\gamma hc$, with some $\gamma$ different from $k$.
\item $u_nx=gb$ or $ic$: $ux$ is then  $lgb$ or $lic$. It cannot be longer as there are no relations of type $\alpha l$. It is moreover not possible to rewrite it in some $\gamma gb$ or $\gamma ic$, with some $\gamma$ different from $l$.
\end{itemize}
This concludes the proof.
\end{itemize}
\end{proof}

 Let us choose any total order extending the following partial order on generators $d < e$, $i < g$,
 $f < h$ and $j < k < l$. We consider the following orientation on relations, obtained from the lexicographic order extending the order on generators:
\begin{align*}
eb & \rightarrow da &
hc & \rightarrow fb \\
gb &\rightarrow ic &
kf &\rightarrow je \\
lg &\rightarrow je &
li &\rightarrow kh.
\end{align*}

The only critical pair of the rewriting system is $lgb$, which is confluent. The obtained rewriting system is convergent: the algebra is thus PBW \cite[Section 3]{GHM}. Let us consider the partition poset $\Pi$ drawn on Figure \ref{PPoset}, which is an interval of one partition poset associated with the algebra. The edges are labelled by the generators on which the upper element must be multiplied to obtain the lower element.

\begin{center}
\begin{figure}
\begin{tikzpicture} 
\node[draw, circle] (A) at (0,0) {A};
\node[draw, circle] (B) at (-2,2) {B};
\node[draw, circle] (S) at (0,2) {S};
\node[draw, circle] (U) at (2,2) {U};
\node[draw, circle] (V) at (-2,4) {V};
\node[draw, circle] (W) at (0,4) {W};
\node[draw, circle] (X) at (2,4) {X};
\node[draw, circle] (Z) at (0,6) {Z};

\draw (A.north west)edge node[fill=white]{$a$}(B.south);
\draw (A.north)edge node[fill=white]{$b$}(S.south);
\draw (A.north east)edge node[fill=white]{$c$}(U.south);

\draw (B.north)edge node[fill=white, near start]{$d$}(V.south);
\draw (S.north)edge node[fill=white, near start]{$f$}(W.south);
\draw (S.north west)edge node[fill=white, near end]{$e$}(V.south east);
\draw (S.north east)edge node[fill=white, near end]{$g$}(X.south west);
\draw (U.north west)edge node[fill=white, near start]{$h$}(W.south east);
\draw (U.north)edge node[fill=white, near end]{$i$}(X.south);
\draw (V.north)edge node[fill=white, near start]{$j$}(Z.south west);
\draw (W.north)edge node[fill=white, near start]{$k$}(Z.south);
\draw (X.north)edge node[fill=white, near start]{$l$}(Z.south east);
\end{tikzpicture}
\caption{Partition poset associated to the algebra \label{PPoset}}
\end{figure}
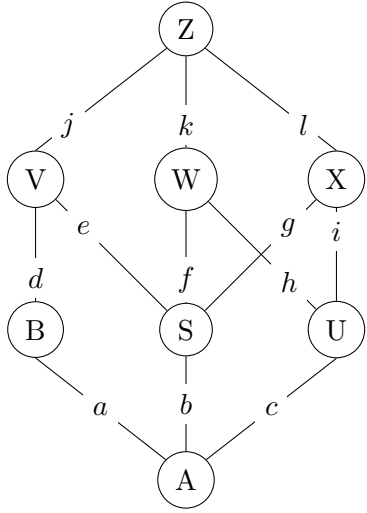
\end{center}

\begin{lem}
The poset $\Pi$ admits no CL-labelling such that the poset is CL-shellable and minimal chains associated with the CL-labelling correspond to PBW elements of the algebra, i.e. to normal forms in the previous rewriting system. 
\end{lem}

\begin{proof}
We prove the lemma by reductio ad absurdum. Suppose that a CL-labelling $l$ of $\Pi$ exists, such that the poset $\Pi$ is CL-shellable and minimal chains in the lexicographic order for the CL-labelling correspond to PBW elements of the algebra. We denote by $\lessdot$ the covering relation associated with the poset $\Pi$. Then in the interval $[A;W]$, the orientation of the rewriting system implies that the chain $A \lessdot S \lessdot W$ must be the smallest one and increasing, thus implies $l(A \lessdot S)<l(A \lessdot U)$. Moreover, in the interval $[A;X]$, the orientation of the rewriting system implies that the chain $A \lessdot U \lessdot X$ must be the smaller one and increasing, thus implies $l(A \lessdot U)<l(A \lessdot S)$. Remark that we do not need to specify to which maximal chains these edges belong as if two maximal chains coincide along their bottom edge, then the labels also coincide on this edge. We thus get a contradiction as we cannot have $l(A \lessdot S)<l(A \lessdot U)<l(A \lessdot S)$: such a CL-labelling cannot exist. 
\end{proof}

Considering an algebra $\mathcal{A}$, and an associated rewriting system $R(\mathcal{A})$. For the same reasons as detailed in the previous proof, every poset associated with $\mathcal{A}$ containing the subposet in Figure \ref{fig5}, with $\beta\alpha $ and $\beta' \alpha'$ normal forms of $R(\mathcal{A})$, cannot admit a CL-labelling such that the poset is CL-shellable and minimal chains in the lexicographic order for the CL-labelling correspond to normal forms of $R(\mathcal{A})$.
\begin{figure}
\begin{tikzpicture}
\coordinate (A) at (-1,0);
\draw (A) node{$\bullet$};
\coordinate (B) at (-2,2);
\draw (B) node{$\bullet$};
\coordinate (S) at (0,2);
\draw (S) node{$\bullet$};
\coordinate (V) at (-2,4);
\draw (V) node{$\bullet$};
\coordinate (W) at (0,4);
\draw (W) node{$\bullet$};

\draw (A.north west)edge node[fill=white]{$\alpha$}(B.south);
\draw (A.north east)edge node[fill=white]{$\alpha'$}(S.south);

\draw (B.north)edge node[fill=white, near start]{$\beta$}(V.south);
\draw (B.north east)edge node[fill=white, near start]{$\gamma$}(W.south west);
\draw (S.north)edge node[fill=white, near start]{$\beta'$}(W.south);
\draw (S.north west)edge node[fill=white, near end]{$\gamma'$}(V.south east);
\end{tikzpicture}
\caption{Poset in which the PBW elements $\beta\alpha $ and $\beta' \alpha'$ cannot correspond to minimal chains of any CL-labelling} \label{fig5}
\end{figure}
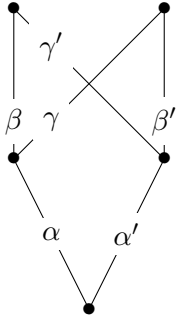

\begin{rque}
In the previous counter-example, exchanging the order between $i$ and $g$ enables to define a CL-shelling satisfying the wanted property.
\end{rque}

\subsection{Second counter-example}

Using the idea of the first example, we construct a counter-example for the converse of Theorem \ref{ThmCLPBW}, i.e. given any PBW basis associated with the algebra, there is no associated isomorphism compatible CL-shelling.

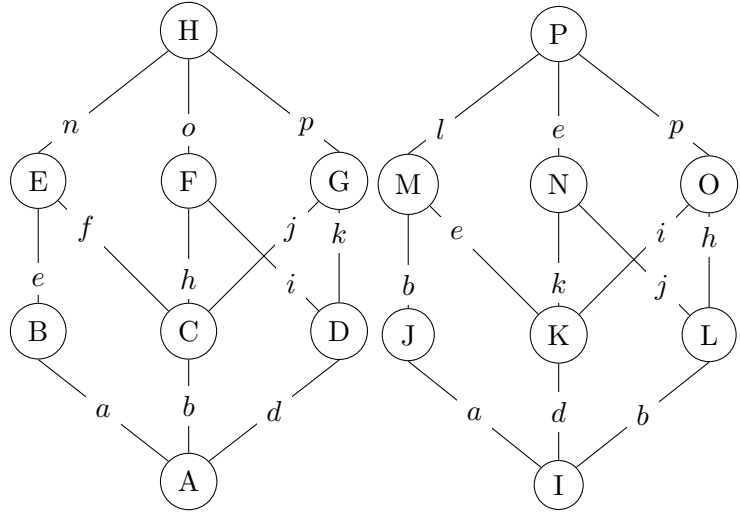
\begin{figure}
\begin{tikzpicture} 
\node[draw, circle] (A) at (0,0) {A};
\node[draw, circle] (B) at (-2,2) {B};
\node[draw, circle] (S) at (0,2) {C};
\node[draw, circle] (U) at (2,2) {D};
\node[draw, circle] (V) at (-2,4) {E};
\node[draw, circle] (W) at (0,4) {F};
\node[draw, circle] (X) at (2,4) {G};
\node[draw, circle] (Z) at (0,6) {H};

\draw (A.north west)edge node[fill=white]{$a$}(B.south);
\draw (A.north)edge node[fill=white]{$b$}(S.south);
\draw (A.north east)edge node[fill=white]{$d$}(U.south);

\draw (B.north)edge node[fill=white, near start]{$e$}(V.south);
\draw (S.north)edge node[fill=white, near start]{$h$}(W.south);
\draw (S.north west)edge node[fill=white, near end]{$f$}(V.south east);
\draw (S.north east)edge node[fill=white, near end]{$j$}(X.south west);
\draw (U.north west)edge node[fill=white, near start]{$i$}(W.south east);
\draw (U.north)edge node[fill=white, near end]{$k$}(X.south);
\draw (V.north)edge node[fill=white, near start]{$n$}(Z.south west);
\draw (W.north)edge node[fill=white, near start]{$o$}(Z.south);
\draw (X.north)edge node[fill=white, near start]{$p$}(Z.south east);
\end{tikzpicture}
\begin{tikzpicture} 
\node[draw, circle] (A) at (0,0) {I};
\node[draw, circle] (B) at (-2,2) {J};
\node[draw, circle] (S) at (0,2) {K};
\node[draw, circle] (U) at (2,2) {L};
\node[draw, circle] (V) at (-2,4) {M};	
\node[draw, circle] (W) at (0,4) {N};
\node[draw, circle] (X) at (2,4) {O};
\node[draw, circle] (Z) at (0,6) {P};

\draw (A.north west)edge node[fill=white]{$a$}(B.south);
\draw (A.north)edge node[fill=white]{$d$}(S.south);
\draw (A.north east)edge node[fill=white]{$b$}(U.south);

\draw (B.north)edge node[fill=white, near start]{$b$}(V.south);
\draw (S.north)edge node[fill=white, near start]{$k$}(W.south);
\draw (S.north west)edge node[fill=white, near end]{$e$}(V.south east);
\draw (S.north east)edge node[fill=white, near end]{$i$}(X.south west);
\draw (U.north west)edge node[fill=white, near start]{$j$}(W.south east);
\draw (U.north)edge node[fill=white, near end]{$h$}(X.south);
\draw (V.north)edge node[fill=white, near start]{$l$}(Z.south west);
\draw (W.north)edge node[fill=white, near start]{$e$}(Z.south);
\draw (X.north)edge node[fill=white, near start]{$p$}(Z.south east);
\end{tikzpicture}
\caption{Second counter-example}
\end{figure}

We consider the algebra on $13$ generators: $a$, $b$, $d$, $e$, $f$,
 $h$, $i$, $j$, $k$, $l$, $n$, $o$, $p$, with relations given as follows.

\begin{align*}
&ba=ed  &ea=fb   \\
&hb=id  &jb=kd  \\
&oi=pk  &ej=ph \\
&le=ek=pi  &nf=oh=pj 
\end{align*}

Once again, we have the following proposition:
\begin{prop}
This algebra is basic-set
\end{prop}

\begin{proof} We follow the strategy of the proof of Proposition \ref{basic1}.

The possible $x$ for a relation $ux=vx$ are $a$, $b$, $d$, $e$, $f$,
$h$, $i$, $j$ and $k$. The possible associated words are:
\begin{itemize}
\item ending in $a$: $lba$ and $nea$
\item ending in $b$: $lfb$, $pjb$, $ejb$, $ohb$ and $phb$
\item ending in $d$: $ohed$, $phed$, $oid$, $pid$, $pkd$, $ekd$, 
$lfed$, $pjed$ and $lejed$
\item ending in $e$: $lle$
\item ending in $f$: $nf$
\item ending in $h$: $lph$ and $oh$
\item ending in $i$: $oi$ and $lpi$
\item ending in $j$: $ej$ and $pj$
\item ending in $k$: $pk$ and $ek$.
\end{itemize} 
These words are of maximal lengths (no letter can be added, which 
could be involved in a relation) and none of these words can lead 
to a relation $ux=vx$, with $u \neq v$.
\end{proof}

This algebra admits at least one PBW basis by considering the following rewriting system, associated with any linear order extending $b < l 
<  e < n < o < p < f$, $h < i$, $k < j$:
\begin{align*}
ed&\rightarrow be &
ek&\rightarrow le \\
fb &\rightarrow ea &
id&\rightarrow hb \\
jb&\rightarrow kd &
oh&\rightarrow nf \\
ph&\rightarrow ej &
pi&\rightarrow le \\
pj&\rightarrow nf &
pk&\rightarrow oi \\
\end{align*}

The lexicographic ordering ensures that the rewriting system is terminating and the only critical pairs are $pid$ and $pjb$ which are confluent: this system is convergent and define a PBW basis as defined in definition \ref{defPBW}.

\begin{lem}
The posets associated with this algebra admit no CL-labelling such that the posets are CL-shellable, with a shelling compatible with poset isomorphisms and minimal chains for the CL-labelling corresponding to PBW elements of the algebra, for any kind of PBW basis. Equivalently, the converse of Theorem \ref{ThmCLPBW} is false.
\end{lem}

\begin{proof}
Let us consider the equations $hb=id$ and $jb=kd$. There are four possible orientations. We will prove that there are no orientation which gives a PBW basis associated with minimal chains for a CL-shelling. Note that the intervals $[I;N]$ and $[A;G]$ (resp. $[I;O]$ and $[A;F]$) are isomorphic, thus the labellings $l_b$ (resp. $l_d$) of edges $[I;L]$ and $[A;C]$ (resp. $[I;K]$ and $[A;D]$) must be the same in any interval beginning with it.
\begin{itemize}
\item If $hb \rightarrow id$ and $kd \rightarrow jb$ or $id \rightarrow hb$ and $jb \rightarrow kd$, one of the posets implies $l_b<l_d$ and the other one $l_d<l_b$, which is impossible.
\item If $hb\rightarrow id$ and $jb \rightarrow kd$, either $oid$ or $pkd$ is the normal form of the poset $[A;H]$.
This implies $ea \rightarrow fb$ so that there is only one normal form in $[A;H]$. 
The same reasoning in $[I;P]$ implies $ba \rightarrow ed$. But then the critical pair $jba$ is not confluent: no PBW basis can satisfy that.
The case $di \rightarrow bh$ and $dk \rightarrow bj$ is symmetric.
\end{itemize}
\end{proof}

\bibliographystyle{alpha}
\bibliography{bibli}

\end{document}